\documentclass[12pt]{article}

\usepackage[margin=1in]{geometry}  
\usepackage{graphicx}              
\usepackage{amsmath}               
\usepackage{amsfonts}              
\usepackage{amsthm}                
\usepackage{color}
\usepackage{tabulary}
\usepackage{slashbox}

\newtheorem{thm}{Theorem}[section]
\newtheorem{lem}[thm]{Lemma}
\newtheorem{prop}[thm]{Proposition}
\newtheorem{cor}[thm]{Corollary}
\newtheorem{conj}[thm]{Conjecture}
\newtheorem{defn}[thm]{Definition}
\newtheorem{rmk}[thm]{Remark}

\DeclareMathOperator{\Id}{Id}

\DeclareMathOperator{\PR}{PSL_2(\RR)}
\DeclareMathOperator{\PC}{PSL_2(\CC)}
\DeclareMathOperator{\ES}{\bf{E}}
\DeclareMathOperator{\HS}{\bf{H}}
\DeclareMathOperator{\PS}{\bf{P}}
\DeclareMathOperator{\CS}{\bf{C}}
\DeclareMathOperator{\Ax}{Axis}
\DeclareMathOperator{\SSS}{\mathcal{S}}
\DeclareMathOperator{\EE}{\mathcal{E}}
\DeclareMathOperator{\FF}{\mathcal{F}}
\DeclareMathOperator{\Aut}{Aut}

\newcommand{\NN}{\mathbb{N}}      %
\newcommand{\ZZ}{\mathbb{Z}}      
\newcommand{\RR}{\mathbb{R}}      
\newcommand{\CC}{\mathbb{C}}      
\newcommand{\DD}{\mathbb{D}}      
\newcommand{\HH}{\mathbb{H}}      

\begin{document}


\title{The Space of Geometric Limits of One-generator Closed Subgroups of $\PR$}

\author{Hyungryul Baik \& Lucien Clavier  \\ 
Department of Mathematics\\
Malott Hall, Cornell University \\
Ithaca, New York 14853-4201 USA}

\maketitle

\begin{abstract}
We give a complete description of the closure of the space of one-generator closed subgroups of $\PR$ for the Chabauty topology, by computing explicitly the matrices associated with elements of $\Aut(\DD)\cong \PR$, and finding quantities parametrizing the limit cases.
Along the way, we investigate under what conditions sequences of maps $\varphi_n:X\to Y$ transform convergent sequences of closed subsets of the domain $X$ into convergent sequences of closed subsets of the range $Y$. 
In particular, this allows us to compute certain geometric limits of $\PR$ only by looking at the Hausdorff limit of some closed subsets of $\CC$.
\end{abstract}

\tableofcontents
\listoffigures

\section{Introduction}
\label{intro}

In \cite{Cha1}, C. Chabauty generalized a result of Mahler about the relative compactness of some sets of lattices of $\RR^n$ to a large class of locally compact groups.

In comparison with the Chabauty topology on the space of all closed subgroups of a locally compact group $G$, the space of all closed subsets of $G$ equipped with the Hausdorff distance is tremendously wilder. For instance, the Chabauty topology of $\RR$ is a closed segment (see for instance Section \ref{ss22}).
Also, in the beautiful paper \cite{PourHubb}, the Chabauty topology of $\RR^2$ is shown to be $S^4$.
In contrast, the space of closed subsets of $\RR$ is the Hilbert cube (see \cite{West}).
For a general exposition to Chabauty topology, we highly recommend \cite{Harpe}. 

The use of the Chabauty topology in the study of Kleinian groups (called geometric limits in this context) is now classical; it has interesting applications in the theory of hyperbolic manifolds; see for instance Chapter 9 of \cite{Thurston} and Section 5.9 of \cite{Marden}.



The authors were motivated by the desire to understand the closure of the faithful discrete type-preserving $\PC$-representations of the fundamental group of the once-punctured torus.
Even if a lot is known about geometric limits in general, it is still a challenge to understand the global space of geometric limits of Kleinian groups as a topological space.
For reasons like the existence of infinite enrichments (see \cite{Hubb2}), it is tremendously difficult to understand that space as a whole.
However, it is possible to attack the problem of explicitly describing geometric limits in a few simpler cases, starting in this paper with the space of one-generator closed subgroups of $\PR$.
This rather simple case already presents some subtle issues arising from the special nature of the Chabauty topology (for definitions of terminologies, see Section 2). 
We will first show that the most natural way to parametrize the space of one-generator closed subgroups of $\PR$ is too naive to give a correct idea of its closure in the Chabauty topology, and we will give a new effective parametrization of this space.
This allows us to compute every possible geometric limit of convergent sequences of one-generator closed subgroups, by simply computing the limit of these parameters.
The main results can be found in Sections \ref{es}, \ref{hs} and \ref{glu}.

Here is now a summary of the paper.

Section 2. Given some locally-compact group $G$, the Chabauty topology on the space $\mathcal{C}(G)$ of all closed subgroups of $G$ is induced by the Hausdorff distance on closed subsets of the one-point compactification of $G$, regarded as a set.
We see, for instance, that $\mathcal{C}(\RR)$ is simply a closed segment.

Section 3. Any element of $\Aut(\HH^2)\cong \PR$ is either elliptic (hence a rotation around some point of $\HH^2$), hyperbolic (i.e. it fixes an axis and acts as a translation on it) or parabolic (it fixes only one point in~$\partial \HH^2$). 
As a result, a first ``naive'' picture of the space of all one-generator subgroups of $\PR$ is obtained by describing subgroups generated by an elliptic element (resp. hyperbolic) thanks to 
 its fixed point (resp. fixed axis) and its order as a rotation around that point (resp. its translation length on that axis). 
This is naive in the sense that the \emph{closure} $\ES$ (resp. $\HS$, $\PS$) of the space of subgroups generated by one single elliptic element (resp. hyperbolic, parabolic) is not the one we would expect from looking at the picture (compare Figures \ref{thisisright} and
\ref{foliations}).
Also it is not clear how $\PS$ is attached to both $\ES$ and $\HS$.

Section 4. We give matrix representations to all elements of $\Aut(\HH^2)$. These matrices take into account the parameters described above, but also other quantities ($\mu$ and $\nu$ in our notations)
that will play a fundamental role in Sections \ref{es}, \ref{hs} and \ref{glu}.

Section 5. For two given metric spaces $X$, $Y$ and a sequence of maps $\varphi_n:X\to Y$,
we investigate under what conditions convergent sequences of closed subsets $F_n$ of $X$ are automatically transformed into convergent sequences of closed subsets $\varphi_n(F_n)$ of $Y$.
See Proposition \ref{redlem}.

Section 6. Using Proposition \ref{redlem}, we can reduce the problem of computing the geometric limits of sequences of one-generator subgroups of $\PR$ to the problem of computing the Hausdroff limits of two families of sequences of closed sets of $\CC$.

Sections 7, 8 and 9. Collecting the informations obtained in Section 6, we draw the correct pictures of $\ES$, $\HS$ and $\PS$, and explain how to glue them together.

Section 10. We provide some ideas and work related to the present paper.

\textbf{Acknowledgements:}
We really appreciate that John H. Hubbard let us know about this problem and explained how we could approach it at the beginning. 
We also thank Bill Thurston for the helpful conversations. For the result in Section \ref{coolprop}, we thank James E. West and Iian Smythe for their encouraging and helpful comments.
The proof of Lemma \ref{oneptcompact} comes from a conversation with Juan Alonso.
We also thank the referee for providing constructive comments.

\pagebreak

\section{Preliminaries}
\label{sect:basics}

\subsection{The Chabauty topology}
\label{ss21}

 The Chabauty topology of a locally compact group is a topology on the space of all its closed subgroups. This topology can be understood via the Hausdorff distance. 
 
\begin{defn}
Let $(X,d)$ be a metric space. For every nonempty subsets $A, B$ of $X$, we define the Hausdorff distance between them as the following:
$$ d_H(A,B) = \max\{ \sup_{a \in A} \inf_{b \in B} d(a,b), \sup_{b \in B} \inf_{a \in A} d(a,b)\} $$
\end{defn} 

Note that $d_H(A,B) = 0$ if and only if the closures of $A$ and $B$ are the same. It is also well-known that $d_H$ defines a metric on the set of all compact subsets of $X$. 
It is compact with the topology induced by $d_H$, whenever $X$ is compact. 

Let $G$ be a locally compact group which is second-countable. 
It is then metrizable as a topological group, i.e. its topology is induced by some left-invariant metric.
$G$ being Hausdorff and locally compact, its one-point compactification $\overline{G}$ is Hausdorff.
Recall that $\overline{G}$ is obtained from $G$ by adding to it some infinity-point $P_\infty$ and declaring that the complements of compact subsets form a basis of neighborhood of $P_\infty$.
The following lemma implies that $\overline{G}$ is actually a metric space.

\begin{lem}
\label{oneptcompact}
Let $X$ be a second-countable, locally compact metric space. Then its one-point compactification $\overline{X}$ is metrizable. 
\end{lem} 
\begin{proof}
 Let $K_n$ be an exhaustion of $X$ by compact subsets (i.e. $K_1 \subset K_2 \subset K_3 \subset \cdots$ is an ascending sequence of compact sets so that $\cup_n K_n = X$) . Consider $B'$ the space of Lipschitz functions on $X$ and consider the norm defined by
\[
\|f \| = \sup_{x \in X} |f(x)| + \sup_{x,y\in X} |f(x) - f(y)| + \sum_{n \ge 1} (2^n\sup_{x \in X \setminus K_n} |f(x)|)
\]
Then define $B = \{ f \in B' : \|f\| < \infty \}$. It is easy to see that $B$ is a Banach space. 

Let $B^{\ast}$ denote its dual. One can embed $\overline{X}$ into $B^{\ast}$ via the map $x \mapsto ev_x$ where $ev_x$ is the evaluation map $f \mapsto f(x)$. We map the infinty-point $P_\infty$ to $0$. The assumptions on $X$ guarantee that $ev_x$ and $ev_y$ are distinct if $x$ and $y$ are distinct points. On $B^{\ast}$, we have the norm 
\[
\|L \| = \sup\{|L(f)| : f \in B, \|f\| = 1\}
\]
In this way, we get an induced metric on $\overline{X}$. It is straightforward to check that the topology induced by this metric agrees with the standard topology on the one-point compactification.
 For example, one needs to show that $x$ is close to $P_\infty$ if and only if $ev_x$ is close to $0$. For $x$ close to $P_\infty$, $x \notin K_n$ for large $n$ and then $\| ev_x \| = \sup \{|f(x)| : f \in B, \|f\| = 1\} < 2^{-n}$. Hence $ev_x$ is close to $0$. For the converse, suppose $\|ev_x\| < \epsilon$ for some $0 < \epsilon < 2^{-N}$. If $x \in K_N$ but not in $K_{N+1}$, then we can have a bump function $f$ supported in some neighborhood of $x$ so that $\|f\| = 1$ but $|f(x)| = 2^{-N}$. Thus $x$ must be outside $K_N$. By taking smaller $\epsilon$, or equivalently taking larger $N$, we conclude that if $\|ev_x\|$ is small, $x$ should be outside most compact sets $K_n$. It remains to show that for $x, y \in X$, $x$ is close to $y$ if and only if $\|ev_x - ev_y\|$ is small. This is even easier than the case near $P_\infty$. The readers are invited to check the details.  
\end{proof} 

Let $F(G)$ be the set of all closed subgroups of $G$.
We simultaneously compactify every element of $F(G)$ by adding the infinity-point $P_{\infty}$ to every one of them.
Denote by $\overline{F}(G)$ the space obtained from $F(G)$ by this simultaneous one-point compactification, and set $\overline{A} = A \cup \{P_{\infty} \}$ for any $A \in F(G)$.
Note that since every closed subgroups of a locally compact group is locally compact, every $\overline{A}$ is Hausdorff.

Then $\overline{F}(G)$ is a compact metric space with the Hausdorff distance $d_H$. 
$\overline{F}(G)$, together with the distance $d_H$, is loosely refered to as the \textit{Chabauty topology} of $G$.

\emph{Notational Remark.} When a sequence of one-point compactified subgroups $\overline{A_n}$ converges to a subgroup $\overline{A}$, and when there is no possible confusion, we simply say that $A_n$ converges to $A$ \textit{in the Chabauty topology}.  

In the context of Kleinian groups, the limit of a convergent sequence in the Chabauty topology is called the \emph{geometric limit} of the sequence. For more details about the difference between the algebraic limit and the geometric limit, consult \cite{Hubb2}.

\subsection{The Chabauty topology of $\RR$}
\label{ss22}
The closed subgroups of $\RR$ are either $\RR$ itself, or generated by a real number. 
Let $G_r$ denote the group generated by $r$, so $G_r =r \ZZ= \{ \ldots, -2r, r, 0, r, 2r, \ldots \}$.
Since $G_r = G_{-r}$, we may always assume that $r \ge 0$.
Note that $G_0$ is the trivial group $\{0\}$.
We would like to study the space of these groups in the Chabauty topology.
As described in the previous section, we perform a simultaneous one-point compactification by adding $\infty$ to these subgroups of $\RR$.
By Lemma \ref{oneptcompact}, $\overline{\RR} = \RR \cup \{\infty\}$ is a metric space with some desired topology.
Let $d$ denote this metric.
The proof of the following lemma provides the way we should think about the Chabauty topology, and it plays an important role throughout the paper. 

\begin{lem}
\label{ChabR}
 $\overline{G_r}$ converges to $\overline{G_0} =  \{0, \infty\}$ as $r \to \infty$. 
\end{lem} 
\begin{proof}
 Note that for any compact subset $K$ of $\RR$, there exists a $M >0 $ such that $G_r - \{0\} \subset K^c$ for all $r \ge M$.
Also, the complements of compact subsets form a basis of neighborhood of $\infty$.

 Let $\epsilon >0$ be arbitrary and $N$ be the $\epsilon$-ball around $\infty$ in $\overline{\RR}$. Let $M >0$ be large enough so that $G_r - \{0\} \subset N$ for all $ r \ge M$. 
For all such $r$, the Hausdorff distance between $\overline{G_r}$ and $\overline{G_0}$ is defined by 
\[
d_H(\overline{G_r},\overline{G_0}) = \max ( \sup_{x \in \overline{G_r} } d(x, \overline{G_0}), \sup_{y \in \overline{G_0}} d(y, \overline{G_r}))
\] 
First, look at $ \sup_{x \in \overline{G_r} } d(x, \overline{G_0})$. When $x = 0$, $d(x, \overline{G_0}) = 0$, and else $d(x,\infty)\leq d(x, \overline{G_0}) \leq \epsilon$ by the choice of $r$.
 The second term $ \sup_{y \in \overline{G_0}} d(y, \overline{G_r})$ is bounded above by $\epsilon$ for the same reason, so $d_H(\overline{G_r},\overline{G_0}) \leq \epsilon$.
Therefore, $\overline{G_r} \to \overline{G_0}$ in the Chabauty topology as $r \to \infty$. 
\end{proof} 

\begin{lem}
\label{ChabR2}
 $\overline{G_r}$ converges to $\overline{\RR}$ as $r \to 0$. 
\end{lem} 
\begin{proof}
This can be proved in an elementary way, by using the same techniques as in Lemma \ref{ChabR}. 
\end{proof}

As a simple corollary of Lemmas \ref{ChabR} and \ref{ChabR2}, the Chabauty topology of $\RR$ is isomorphic to the closed interval $[0,\infty]$.

Note that in those lemmas, we do not actually need to know explicitly the metric $d$. This fact will be also used in the proof of the Reduction Lemma (Proposition \ref{redlem}).

\subsection{Objects we are dealing with}

 We use the notations introduced in Subsection \ref{ss21}.
Let $\CS(\PR)$, or simply just $\CS$, be the closure in $\overline{F}(\PR)$ of the set of all one-point compactified \textit{cyclic} subgroups of $\PR$.
Our goal throughout this paper is to present a complete description of $\CS(\PR)$.

Note, after identification of $\PR$ and $\Aut(\HH^2)$, that each element of $\PR$ acts on $\HH^2$.
Let us recall that the isometries of $\HH^2$ are of three types:
elliptic if they have one fixed point inside of $\HH^2$,
hyperbolic if they have two fixed points in the boundary $S^1_\infty=\partial \HH^2$, and
parabolic if they have one fixed point in $S^1_\infty$.
It will sometimes be useful to consider the neutral element of $\PR$ to be of either of the three types above.

As we will see in the next section, most intuition can be gained from the careful observation of the action on $\HH^2$ of the generators of the cyclic closed subgroups of $\PR$.

\pagebreak

\section{Overview for $\CS(\PR)$} 
\label{overview}

\subsection{Elliptic generators } 
\label{elloverview}

We will first study the space of closed subgroups of $\PR$ with one elliptic generator.
Elliptic isometries of $\HH^2$ are rotations around a point in the interior of $\HH^2$. Let $\Omega$ be a subgroup of $\PR$ generated by one elliptic element $e$.
For $\Omega$ to be closed, $e$ needs to have finite order.
Also, note that $\Omega$ is uniquely determined by the choice of the center of the rotation $e$ and the order of $e$.
If we think of $\HH^2$ as the Poincar\'{e} disk $\DD$, then the space of choices for the center of the rotation can be identified with the unit open disk $D$. 
Thus we can express the space $\ES$ of the closed subgroups of $\PR$ with one elliptic generator by
$$ \bigsqcup_{n \ge 2} D_n$$
where the underlying set of $D_n$ is just the unit open disk $D$.
A point in $D_n$ represents the subgroup of all rotations of order $n$ around the corresponding point in $D$.
Of course, $\ES$ is an open subset of $\CS$, we would like to understand its boundary in $\CS$.

It is easy to prove (using a direct proof for instance) that if some moving elliptic generator stays in a finite number of $D_n$ (i.e. its order as a rotation is bounded) while its fixed point tends to a point in the boundary $\partial D$, then the subgroup it generates in $\PR$ tends to the trivial group $\{\Id\}$ for the Chabauty topology.
Thus, part of the closure in $\CS$ of the space of the closed subgroups of $\PR$ with one elliptic generator looks like a wedge sum of countably-many 2-spheres.

Also, it is easy to prove that if the order of the moving elliptic generator increases to infinity while its fixed point stays the same (or tends to some point in the interior of $D$), then the subgroup it generates in $\PR$ tends in the Chabauty topology to the group of all rotations around that point.
Thus, part of the wedge sum described above has to accumulate to some open disk $D_{\infty}$, where a point in $D_\infty$ represents the subgroup of all rotations around the corresponding point in $D$.

For now, it is not quite clear what is happening in the case where our moving generator tends to some point in $\partial D$ and its order tends to infinity. 
It is reasonable to think that the subgroup it generates will converge to some subgroup of parabolic elements, but it is not obvious at the moment what the picture really looks like.

\subsection{Hyperbolic generators} 
\label{hypoverview}

Hyperbolic isometries of $\HH^2$ have two fixed points on the circle at infinity $S^1_{\infty}$.
The geodesic in $\HH^2$ connecting the fixed points of an hyperbolic element $h$ of $\PR $ is called the axis of $h$, denoted $\Ax(h)$.
$h$ acts on $\Ax(h)$ as a translation.
Let us fix some axis. If we consider all subgroups of $\PR$ whose elements are hyperbolic elements sharing this axis, the situation is similar to the case of subgroups of $\RR$: 
we can parametrize these subgroups by the translation length on the axis.

\begin{defn} Let $h \in \PR$ be a hyperbolic isometry of $\HH^2$.
The translation length of $h$ is the distance $d(x, h(x))$ where $d$ is the Poincar\'{e} metric on $\HH^2$ and $x$ is an arbitrary element on $\Ax(h)$. 
\end{defn} 


\begin{lem} Let $(h_n)$ be a sequence of hyperbolic isometries sharing the same axis $l$. Then the following holds:
\label{ChabHyp}
\begin{itemize}
\item[1.] If the translation length of $h_n$ tends to infinity, the limit in the Chabauty topology of the subgroup that $h_n$ generates is the trivial group. 
\item[2.] If the translation length of $h_n$ tends to zero, the limit in the Chabauty topology of the subgroup that $h_n$ generates is the subgroup of $\PR$ of all hyperbolic elements sharing the axis $l$.
\end{itemize}
\end{lem} 
\begin{proof}
See the proof of Lemma \ref{ChabR}.
Note that the actual ``direction'' of the translation is not important, since for any $g \in \PR$, $g$ and $g^{-1}$ generate the same group.
\end{proof}

Therefore, for each axis $l$, the space of subgroups of $\PR$ which contain only hyperbolic elements with axis $l$ is homeomorphic to a closed interval $[0,\infty]$ in the Chabauty topology, where $\infty$ is identified with the trivial group $\{\Id\} \in \PR $, $t\in (0,\infty)$ is identified with the subgroup of $\PR$ generated by an element with translation length $t$ and axis $l$, and $\infty$ is identified with the subgroup of all hyperbolic elements with axis $l$.

The choice of an axis is the same as the choice of two distinct points on the circle.
Thus the space of all those choices can be identified with 
\[
(S^1 \times S^1 - \Delta ) / (x,y) \sim (y,x) 
\]
where $\Delta $ is the diagonal $\Delta = \{ (x, x) \in S^1 \times S^1 \, | \, x \in S^1\}$.
The next figure shows how to see this space as an open M\"{o}bius band.
In order to give a planar representation of this M\"{o}bius band, we replaced the circles $S^1$ above by segments $[0,2\pi]$ in the obvious way.

\begin{figure}[ht]
\label{mobiusband}
\begin{center}
\includegraphics[scale=0.6]{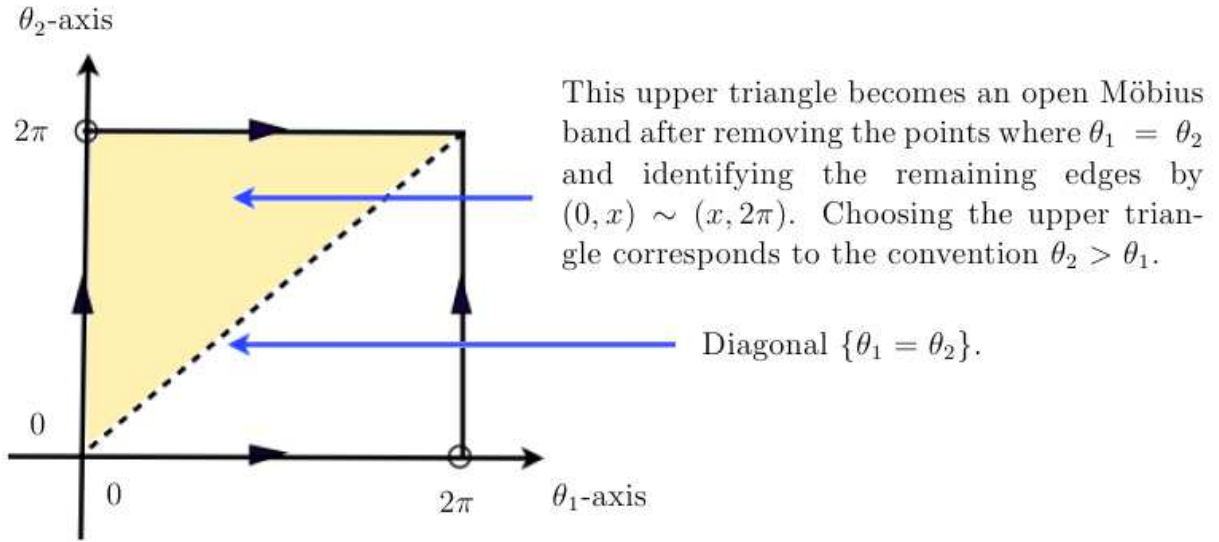}
\end{center}
\caption[The space of choices of axes for hyperbolic isometries]{The space of choices of an axis is an open M\"{o}bius band.}
\end{figure}

Therefore, the space of all cyclic subgroups generated by one hyperbolic element is a cone on the open M\"{o}bius band (see Figure \ref{thisisright}).
 We would like to understand its closure $\HS$ in $\CS$.

\begin{figure}[ht]
\begin{center}
\includegraphics[scale=0.5]{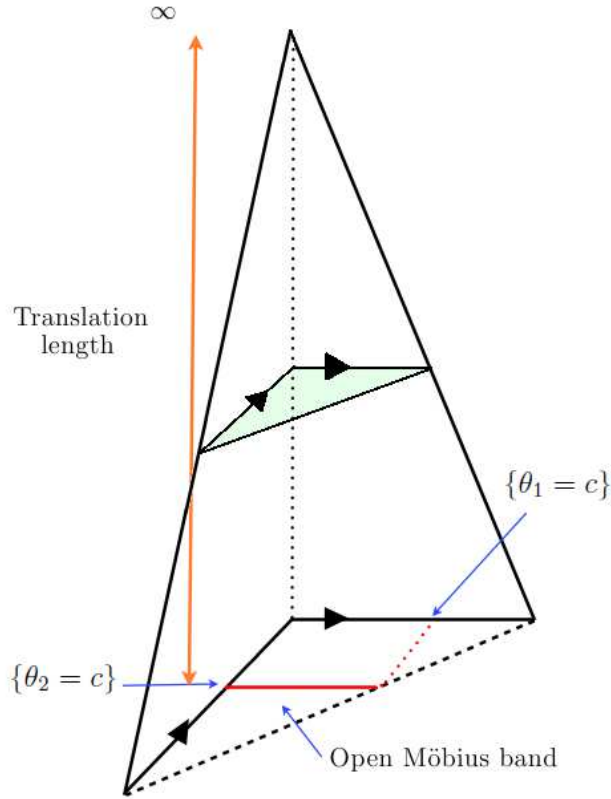}
\end{center}
\caption[The space of subgroups with one hyperbolic generator (naive version)]{The space of subgroups generated by one hyperbolic element is a cone on the open M\"{o}bius band. The left wall $\{\theta_1=0\}$ and the rear wall $\{ \theta_2=2\pi\}$ are identified following the black arrows. In pale green is a slice of constant translation length.}
\label{thisisright}
\end{figure}

As in the case of elliptic generators, it is possible to prove directly that if some hyperbolic generator moves within some horizontal slice (in pale green in the picture) and tends to the boundary of this slice, then the subgroup it generates tends to the trivial group in the Chabauty topology. 

Therefore, this picture, presenting a ``naive'' parametrization of the subgroups generated by one hyperbolic element, is rather deceiving, for the whole wall $\{\theta_1=\theta_2\}$ should be collapsed to a segment.
Also, it is not clear what happens when we approch the base of this wall. This will be settled in the following sections.

\subsection{Parabolic generators} 
\label{paraoverview}

Parabolic elements of $\PR$ have exactly one fixed point on the circle at infinity.
Thus the space of choices of the fixed points is simply $S^1$.
We want to parametrize all the parabolic elements which have the same common fixed point.
In the case of hyperbolic elements, there was a canonical way to express the amount of translation along their axes.
In the parabolic case, there is no such convenient parameters, and we will use in Subsection \ref{ParaMat} a less natural normalization, consisting on controlling the behaviour of more points than just the fixed point. 

Nevertheless, we will see that the space of all subgroups generated by one parabolic element sharing the same fixed point is the same as the Chabauty topology of $\RR$, namely $[0,\infty]$, where $\infty$ represents the trivial group, and 0 represents the subgroup of all parabolic elements around this fixed point.

Therefore, the closure $\PS$ of the space of all subgroups generated by one parabolic element is the cone on a circle, with top vertex representing the trivial group.

\subsection{How do $\ES$, $\HS$, $\PS$ fit together?}

Understanding the boundaries of $\ES$, $\HS$ and $\PS$, and showing how they fit together is the main goal of the rest of the paper.
We will show that both $\ES$ and $\HS$ accumulate to $\PS$ but not to each other.

In all the next sections (except Section \ref{coolprop} which deals with more general spaces than just $\PR$), we will analyze carefully the space of subgroups with one generator.
This will consist in a series of parallel arguments.
We will usually talk about the elliptic generator case first, since it is simpler and somewhat enlightening for the second case, namely the case of hyperbolic generators. The case of parabolic generators will always come last.

The first step in this analysis is to represent every elleptic, hyperbolic and parabolic elements as $2 \times 2$ matrices.


\pagebreak

\section{The key proposition } 
\label{coolprop}

Now we state the Reduction Lemma, which allows one to transform convergent sequences in one space to convergent sequences in a different space, for the Hausdorff topology.
This proposition is stated in a greater generality than we actually need in this paper, but we believe it is interesting in its own right. 
We do not claim that this result is new; nevertheless, we could not find it anywhere in the literature. 

\begin{prop}[Reduction Lemma]
\label{redlem}
Let $(X,d_X),\, (Y,d_Y)$ be two second countable, locally compact metric spaces. 
Let $( \varphi_n)$ be a sequence of maps from $X$ to $Y$, converging to a continuous proper map~$\varphi$, uniformly on every compact subset.
Assume that for every compact subset $K \subset Y$, the closed subset 
\[
\overline{\bigcup_{n\geq N}\varphi_n^{-1}(K)}
\]
is compact for $N$ large enough.

Then whenever a sequence of closed subsets $F_n\subset X$ converges to a closed subset $F$ in the Hausdorff topology of $X$, the subsets $\overline{\varphi_n(F_n)}$ converge to $\overline{\varphi(F)}$ in the Hausdorff topology of $Y$.
\end{prop}

\begin{proof}
 It is possible to prove this directly, using a so-called ``$\epsilon$/$\delta$'' argument.
 Since we would like to highlight the meaning of the condition that
$
\overline{\bigcup_{n\geq N}\varphi_n^{-1}(K)}
$
is compact for $N$ large enough, let us use a slightly different approach. 
First, notice that the case where $X$ is compact is more or less immediate.
\begin{lem}
\label{fact1}
 Let $(X,d_X),\, (Y,d_Y)$ be two metric spaces, $X$ compact, $Y$ locally compact. 
Let $(\varphi_n)$ be a sequence of maps from $X$ to $Y$, converging uniformly to a continuous map $\varphi$. 
Then whenever a sequence of closed subsets $F_n\subset X$ converges to a closed subset $F$ in the Hausdorff topology of $X$, the subsets $\overline{\varphi_n(F_n)}$ converge to $\overline{\varphi(F)}$ in the Hausdorff topology of~$Y$.
\end{lem}
\begin{proof}
This fact can be proved by a simple $\epsilon$/$\delta$ argument left to the reader. 
\end{proof}

Back to the proof of Proposition \ref{redlem}: if $Y$ is compact, $\varphi$ being proper implies that $X$ is compact. Thus, let us suppose that neither $X$ nor $Y$ are compact. 
We would like to reduce the problem to the previous case, where the spaces were compact.
Thus, consider $\widetilde{X}$ (resp. $\widetilde{Y}$) the one-point compactification of $X$ (resp. $Y$), and define 
\[
\widetilde{\varphi_n},\; \widetilde{\varphi} : \widetilde{X} \rightarrow \widetilde{Y}
\]
to be those extensions of $\varphi_n$, $\varphi$ that send the infinity point of $\widetilde{X}$ to the infinity point of $\widetilde{Y}$.

\begin{lem}
\label{fact2}
$\widetilde{\varphi_n}$ converges uniformly to $\widetilde{\varphi}$.
\end{lem}
\begin{proof}
The condition concerning the compactness of
$
\overline{\bigcup_{n\geq N}\varphi_n^{-1}(K)}
$
means exactly that for every neighborhood $\mathcal{N}_2$ of infinity in $\widetilde{Y}$, there is a little neighborhood $\mathcal{N}_1$ of infinity in $\widetilde{X}$ which is sent into $\mathcal{N}_2$ by all $\widetilde{\varphi_n}$, for $n$ larger than some integer $N$.

Indeed, using that complements of compact subsets form a basis of infinity, the latter statement can be rewritten:
\[ \begin{array}{c}
 \forall K_2 \subset Y \text{ compact, } \exists K_1 \subset X \text{ compact, } \exists N \in \NN \text{ s.t. } \forall n\geq N, \\
 \varphi_n(K_1^c) \subset K_2^c  
\end{array}\] 
but since $\varphi_n(K_1^c) \subset K_2^c $ is equivalent to $\varphi_n^{-1}(K_2) \subset K_1 $, it can be rewritten as 
\[ 
 \forall K_2 \subset Y \text{ compact, } \exists K_1 \subset X \text{ compact, } \exists N \in \NN \text{ s.t. } \bigcup_{n\geq N}\varphi_n^{-1}(K_2) \subset K_1 
\] 
i.e.
\[
  \forall K_2 \subset Y \text{ compact, } \exists N \in \NN \text{ s.t. } \overline{\bigcup_{n\geq N}\varphi_n^{-1}(K_2)} \text{ is compact}
\]

Now, for every $\epsilon>0$ we can choose $\mathcal{N}_2$ to be the ball of radius $\epsilon/2$ around the infinity point of $\widetilde{Y}$, $K_1$ and $N\in \NN$ like above.
Taking a larger $K_1$ if necessary, we can always assume that $\widetilde{\varphi}$ also sends $\mathcal{N}_1=K_1^c$ into $\mathcal{N}_2$.
But then, for all $n\geq N$ and all $x\in K_1^c$, 
\[
 d_{\widetilde{Y}}(\widetilde{\varphi_n}(x), \widetilde{\varphi}(x))\leq  d_{\widetilde{Y}}(\widetilde{\varphi_n}(x), \infty)+ d_{\widetilde{Y}}(\widetilde{\varphi}(x), \infty) \leq \epsilon  
\]

Additionally, since $\widetilde{\varphi_n}$ converges to $\widetilde{\varphi}$ uniformly on $K_1$, we can replace $N$ by a larger integer if necessary, and assume that for all $n\geq N$ and all $x\in K_1$, 
\[
 d_{\widetilde{Y}}(\widetilde{\varphi_n}(x), \widetilde{\varphi}(x))\leq  \epsilon  
\]

Therefore, $\widetilde{\varphi_n}$ converges uniformly to $\widetilde{\varphi}$.
\end{proof}

Back to the proof of Proposition \ref{redlem}, since $\widetilde{\varphi_n}$ converges uniformly to $\widetilde{\varphi}$ and $\widetilde{\varphi}$ is continuous (because $\varphi$ is continuous and proper), we are done by Lemma \ref{fact1}.
\end{proof}

\begin{rmk}
 We can actually show more than Lemma \ref{fact2}: in the case where both $X$ and $Y$ are non-compact, the hypotheses of Proposition \ref{redlem} are \underline{equivalent} to asking that $\widetilde{\varphi_n}$ converges uniformly to $\widetilde{\varphi}$ and that $\widetilde{\varphi}$ is continuous.
\end{rmk}

\begin{rmk}
\label{rmkomega}
The proof of Proposition \ref{redlem} actually shows the following: suppose the functions $\varphi_n$ are only defined on some domains $\Omega_n \subset X$ satisfying that for any compact subset $K\subset X$, we can find an integer $N$ such that for all $n\geq N$, $K\subset \Omega_n$. Equivalently, $\Omega_n^c\subset \mathcal{N}$ for every neighborhood $\mathcal{N}$ of the infinity-point $\infty\in \overline{X}$ and for all $n$ large enough. Then the conclusion of Proposition \ref{redlem} still holds \textit{if $F_n\subset \Omega_n$ for every $n$}, simply by declaring that $\widetilde{\varphi_n}$ sends every point of $\Omega_n$ to $\infty\in \widetilde{Y}$.
\end{rmk}

We would like to finish this section with some (counter) examples.

In Proposition \ref{redlem}, it is not necessary for the $\varphi_n$ to be proper (nor continuous, actually), as we can see by defining $X=Y=\RR$ and
\[
\varphi_n(x)= 	\begin{cases}
			-n \text{ if } x\leq -n \\
  			x \text{ if } |x|\leq n\\
			n \text{ otherwise }
 		\end{cases}
\]
which are bounded, but converge uniformly on every compact subset to the identity map, and Proposition \ref{redlem} still applies.

Emphatically, it is not sufficient either that the $\varphi_n$ be proper and continuous, as we see by defining $X=Y=\RR$ and
\[
\varphi_n(x)= 	\begin{cases}
  			x \text{ if } x\leq n\\
			\frac{3n-x}{2} \text{ otherwise }
 		\end{cases}
\]
which are continuous, proper and converge on every compact subset to the identity map~$\varphi$.
But, choosing $F_n=F=\ZZ$, we have $1/2\in \varphi_n(F_n)$ for each $n$ (take $x=3n-1$), but $1/2\notin \varphi(F)=\ZZ$.

In addition, $\varphi_n$ being proper does not imply that $\varphi$ is: define $X=Y=\RR$ and
\[
\varphi_n(x)= 	\begin{cases}
  			x+n \text{ if } x\leq -n\\
			x-n \text{ if } x\geq n\\
			0 \text{ otherwise }
 		\end{cases}
\]
which are continuous and proper, but converge uniformly on every compact subset to the zero map.

\pagebreak

\section{Matrix representations} 
\label{matrix}

In this section, we show how to represent every elliptic, hyperbolic and parabolic element of $\PR$ as a $2 \times 2$ matrix.

\begin{rmk}
It will usually be more convenient to use the Poincar\'{e} disk model $\DD$ instead of the upper half-plane model $\HH^2$, for symmetry reasons. 
As a result, the matrices we are interested in will have complex entries, the reader should not be surprised by this.
See for instance \cite{Hubb} for the standard identification of $\PR$ and $\Aut(\DD)$.
\end{rmk}



\subsection{The elliptic case} 

Let $E_{p, \phi}$ denote the elliptic element which is a rotation around $p \in \DD$ with angle $\phi$.
We use the polar coordinates $p = r e^{i \theta}$ for elements of $\DD$.
We have the following lemma. 

\begin{lem} 
\label{ellmatrix} 
The element in $\Aut(\DD)$ that corresponds to the M\"{o}bius transformation $E_{p, \phi}$ is represented by the matrix 
$$ e^{- i \phi / 2} \begin{pmatrix} 1 - \mu & p \mu \\ -\overline{p} \mu & 1 + p\overline{p} \mu \end{pmatrix} $$ 
where $\mu = \frac{1 - e^{i \phi}}{1- |p|^2}$.
\end{lem} 
\begin{proof}
$f := [z \mapsto \frac{z - p}{1 - \overline{p}z}]$ is an automorphism of  $\DD$ which maps $p$ to $0$. If $R_{\phi}$ denotes the rotation around $0$ by $\phi$ in $\DD$, then $E_{p, \phi} = f^{-1} \circ R_{\phi} \circ f$, which is represented by the matrix $$\begin{pmatrix} e^{i \phi} - p\overline{p} & p(1-e^{i \phi}) \\ - \overline{p}(1-e^{i \phi}) & 1 - p \overline{p} e^{i \phi} \end{pmatrix}$$ Its determinant is $e^{i \phi} ( 1 - p\overline{p})^2$, so the result follows. 
\end{proof} 

\begin{rmk} 
Let $<E_{p,\phi}>$ denote the subgroup generated by $E_{p,\phi}$. When $\phi$ is an irrational angle, $<E_{p,\phi}>$ is not closed, so we can ignore this case.
From now on, we will assume $\phi$ has finite order.
Also, we can always replace $E_{p,\phi}$ by $E_{p, 2\pi / |\phi|}$, where $|\phi|$ is the order of $\phi$, since they both generate the same group.
 This observation will be useful later. 
\end{rmk}

\subsection{The hyperbolic case} 

Choosing a hyperbolic element amounts to choosing an unordered pair of distinct elements in $S^1 = \partial \DD$ and a translation length. 
 
$\bf{Convention}$  When we choose an unordered pair of distinct elements $e^{i \theta_1}$, $e^{i \theta_2}$ in $S^1$, we always pick the labeling $\theta_1\neq\theta_2\in [0,2\pi]$ so that $\Delta := \theta_2 - \theta_1 \in (0, 2 \pi)$.

For $\theta_1$, $\theta_2$ as explained above, and for $a>1$,
let $H_{\theta_1, \theta_2, a}$ denote the hyperbolic element with translation length $\log a$ whose repelling and attracting fixed points are respectively $e^{i \theta_1}$ and $e^{i \theta_2}$.

\begin{rmk}
Since $H_{\theta_1, \theta_2, a}$ and its inverse generate the same group, it is sufficient to consider that $a>1$. 
\end{rmk}

We have the following lemma as an analogue of Lemma \ref{ellmatrix}. 
\begin{lem} 
\label{hypmatrix}
The element in $\Aut(\DD)$ that corresponds to the M\"{o}bius transformation $H_{\theta_1, \theta_2, a}$ is represented by the matrix
$$ \frac{1}{\sqrt{a}}  \begin{pmatrix}1 - \nu e^{i\Delta}  & \nu e^{i \theta_2} \\ - \nu e^{-i \theta_1} & 1 + \nu \end{pmatrix} ,$$ 
where $\nu = \frac{a-1}{1-e^{i\Delta}}$. 
\end{lem}
\begin{proof} 
$\phi = [z\mapsto -e^{i (\theta_2 - \theta_1)/2} \dfrac{z - e^{i \theta_1}}{z - e^{i \theta_2}}]$ is the isometry between $\DD$ and $\HH$ which maps $e^{i \theta_1}$ to 0, $e^{i \theta_2}$ to $\infty$ and $(\theta_1+\theta_2)/2$ to 1.
Then on $\HH^2$, the hyperbolic element with translation length $\log a$ and repelling, attracting fixed points $0$ and $ \infty$ is simply $[z \mapsto a z]$.
Thus we have $H_{\theta_1, \theta_2, a} = \phi^{-1} \circ [z \mapsto a z] \circ \phi$.
The result follows by direct computation. 
\end{proof}

\begin{rmk}
As we will see in Section \ref{limchab}, the parameters $\mu$ and $\nu$ (or rather their modulus) are in fact fundamental, since they express \textit{in a quantative way} how sequences of subgroups generated by one elliptic (resp. hyperbolic) element can converge to subgroups generated by one parabolic element. This is the ingredient we needed to understand clearly the boundary of $\CS$. See Sections \ref{limchab}, \ref{es}, \ref{hs} and \ref{glu}.
\end{rmk}
 

\subsection{The parabolic case} 
\label{ParaMat}
 In the case of parabolic isometries of $\HH^2$, there is no well-defined notion of translation length.
Indeed, all parabolic element in $\Aut(\HH^2)$ fixing $\infty$ have matrix representations of the form $\begin{pmatrix} 1 & b \\ 0 & 1\end{pmatrix}$. But these are all conjugate by dilatations $\begin{pmatrix} a & 0 \\ 0 & 1/a\end{pmatrix}$.
Thus to parametrize each group consistently, we need a normalization which we describe now. 
 
$\bf{Normalization}$ 
As explained above, we  want to express every parabolic isometry of the Poincar\'{e} disk $\DD$ as a translation $[ z \mapsto z + b]$ in $\HH^2$.
To do this in a consistent way, we are going to ask the parabolic element of $\Aut(\DD)$ to be conjugate with the translation $[ z \mapsto z + b]$ via a map
$f : \DD \rightarrow \HH^2$ sending $-e^{i \theta}$ to 0, 0 to $i$, and $e^{i \theta} $ to $\infty$.
Then we see that $b$ is uniquely defined, and that $f$ is simply the map
\[
f(z) = -i \dfrac{z + e^{i \theta}}{z - e^{i \theta}} 
\]
Then the following holds: 

\begin{lem}
 Define $P_{\theta, \rho}$ to be the matrix 
\[
\begin{pmatrix} 1 - i \rho & i \rho e^{i\theta} \\ -i \rho e^{-i \theta} & 1 + i \rho \end{pmatrix} 
\]
Then $P_{\theta, \rho}$ represents the translation $[z \mapsto z - 2 \rho]$ in $\HH^2$ under the above normalization. 
\end{lem}
\begin{proof}
It is straightforward to check that $P_{\theta, \rho}$ and 
$$\begin{pmatrix} -e^{i \theta} & i e^{i \theta} \\ -1 & -i \end{pmatrix} \begin{pmatrix} 1 & -2\rho \\ 0 & 1 \end{pmatrix} \begin{pmatrix} -i & -i e^{i \theta} \\ 1 & -e^{i \theta}  \end{pmatrix}$$
differ from a scalar. 
Therefore the matrix $P_{\theta, \rho}$ represents the transformation $f^{-1} \circ [z \mapsto z - 2\rho] \circ f$ of $\DD$.
\end{proof} 

Using this normalization, we can unambiguously speak about the ``translation distance'' of a given parabolic element of $\PR$.

\pagebreak

\section{Limits in the Chabauty topology }
\label{limchab}

Here we show how to use the matrices obtained in Section \ref{matrix} to deduce the possible limits of subgroups in the Chabauty topology.
We start with the elliptic case. 

\subsection{The elliptic case } 

We consider a sequence $\SSS=\{\SSS_n\}$, where for each $n\in \NN$
$$\SSS_n = <E_{p_n,\phi_n}>=\{(E_{p_n,\phi_n})^k\,|\, k\in \ZZ\}$$
Recall that $<E_{p_n,\phi_n}> = <E_{p_n, 2\pi / \omega_n}>$, where $\omega_n=|\phi_n|$ is the order of $\phi_n$.
We will show in the next proposition that the limit of the sequence $\SSS$ in the Chabauty topology is governed by the limits of $r_n=|p_n|$, $\omega_n=|\phi_n|$ and 
$\rho_n=|\mu_n|$ where $\mu_n = \frac{1-e^{2\pi i /\omega_n}}{1-|p_n|^2}$.
Remark that, since the space of all closed subgroups is compact for the Chabauty topology, extracting a subsequence if necessary,
we can always assume that these sequences converge.
The different limits $\SSS$ can have are summarized in Proposition \ref{elllimchab} below. 

\begin{prop} 
\label{elllimchab}
The following table shows the Chabauty limit of $\SSS$ according to the possible limits of $r_n$ and $|\phi_n|$. 

\begin{tabular}{ *{3}{|c}|} 
   \hline
   \backslashbox{$p_\infty$}{$\omega_\infty$} & $< \infty$ & $\infty$  \\
   \hline
   $|p_\infty|\neq 1$ & $<E_{p_\infty, 2\pi/\omega_\infty}>$ & subgroup of all rotations around $p_\infty$ \\
   \hline
   $|p_\infty|=1$ & $\{1\}$ & $<P_{p_\infty, \rho_\infty}>$ \\
   \hline
  \end{tabular} \\
Here the notations are $p_\infty=\lim p_n \in \overline{\DD}$, $\omega_\infty=\lim \omega_n\in \NN\cup \infty$ and $\rho_\infty=\lim \rho_n\in [0,\infty]$.
By convention, for every $p\in S^1$, $<P_{p, \infty}>$ is the trivial group and  $<P_{p, 0}>$ is the subgroup of all parabolic elements fixing $p$ (this convention will be justified in Subsection \ref{pcase}). 
\end{prop} 

\begin{proof}
Since the most interesting case is when $r_n=|p_n| \to 1$ and $|\phi_n| \to \infty$, let us first assume we are in that case.

Set $\EE_n := \{ \dfrac{1-e^{i k \phi_n}}{1- r_n^2}\,|\, k \in
\ZZ\}$, and define the maps $\varphi_n, \varphi: \CC \to \PC$ by 
\begin{align*}
\varphi_n(z) &=(1-(1-r_n^2)z)^{-1/2} \begin{pmatrix} 1 - z & p_n z \\
-\overline{p_n} z & 1 + r_n^2 z \end{pmatrix}  \\
\varphi(z) &= \begin{pmatrix} 1 - z & p_\infty z \\ -\overline{p_\infty} z & 1
+ z \end{pmatrix} 
\end{align*}
One should note that the members of $\EE_n$ are the complex numbers $``\mu"$ for the matrices $E^k_{p_n, \phi_n}$. 

Then by construction $\SSS_n=\varphi_n(\EE_n)$, and
$\varphi$ is continuous and proper. The following lemmas show that we can apply the Reduction Lemma.

\begin{lem}
\label{lemell1}
$\varphi_n$ converges to $\varphi$ uniformly on every compact subset.
\end{lem}
\begin{proof}
It is sufficient to prove it for every compact $K_R=\{z\in \CC;\; |z|\leq R\}$.
Fix some $\epsilon>0$.
Since $r_n\to 1$, we can find for every $R>0$ some integer $N$ such that, for all $n\geq N$ and all $z\in K_R$,
\[
 1-\epsilon\leq |1-(1-r_n^2)z|^{-1/2}\leq 1+\epsilon
\]
Therefore, we can also find an integer such that for every $n$ larger than this $N$,
\[
 \|\varphi_n(z)-\varphi(z)\|_\infty \leq \epsilon
\]
holds for every $z\in K_R$. 
\end{proof}

\begin{lem}
\label{lemell2}
 For any compact subset $K$ of $\PR$, the closed subset of $\CC$
\[
\overline{\bigcup_{n\geq N}\varphi_n^{-1}(K)}
\]
is compact for $N$ large enough.
\end{lem}
\begin{proof}
It is sufficient to prove that for every $R>0$ and for every $z$ with $|z|>R$, one of the entries of $\varphi_n(z)$ has a modulus greater that some quantity $A(R)$ depending only on $R$, with $A(R) \to \infty$ as $R\to \infty$.
But this is clear, because the first entry of $\varphi_n(z)$ has modulus
\[
 |1-(1-r_n^2)z|^{-1/2} |1-z|\geq \frac{R-1}{\sqrt{R+1}}
\]
\end{proof}

Therefore, we can apply the Reduction Lemma. Thus, whenever $\EE_n$ converge to some closed set $\EE$ in the Hausdorff topology of $\CC$, then the sequence $\SSS_n=\varphi_n(\EE_n)$ converges to $ \varphi(\EE)$ in the Chabauty topology. 

\begin{rmk}
 The functions $\varphi_n$ are not actually defined for $z=1/(1-r_n^2)$.
But since $r_n \to 1$, we can apply Remark \ref{rmkomega} with $\Omega_n=\CC \setminus \{1/(1-r_n^2)\}$.
\end{rmk}

The final piece of information we need in order to conclude is the following lemma. 
\begin{lem} In the Hausdorff topology on $\CC$, the sequence of sets $\EE_n$ converges to the set $\EE=\{k i \rho_{\infty} ; k \in \ZZ\}$, 
where $\rho_{\infty} = \lim_{n \to \infty} |\mu_n|$. 
\end{lem} 
\begin{proof}
Note that $\lim_{n \to \infty} \mu_n = i \rho_{\infty}$. This can easily be proved in a direct manner, see Figure \ref{en} for geometric intuition. 
\end{proof}

Note that the image of the set $\{k i \rho_{\infty} \,|\, k \in \ZZ\}$ under $\varphi$ is 
\[
\left\{
\begin{pmatrix} 1 - k i \rho_{\infty} & k i \rho_{\infty} p_{\infty} \\ -k i \rho_{\infty} \overline{p_{\infty}} & 1 + k i \rho_{\infty} \end{pmatrix} \,|\, k \in \ZZ
\right\} 
\]
which is exactly the subgroup generated by $P_{p_\infty, \rho_\infty}$.

Thus we are done for the case where $r_n=|p_n| \to 1$ and $|\phi_n| \to \infty$. 

The other cases, easier and similar, are left to the reader.
\end{proof} 



\begin{figure}[ht]
\begin{center}
\includegraphics[scale=0.3]{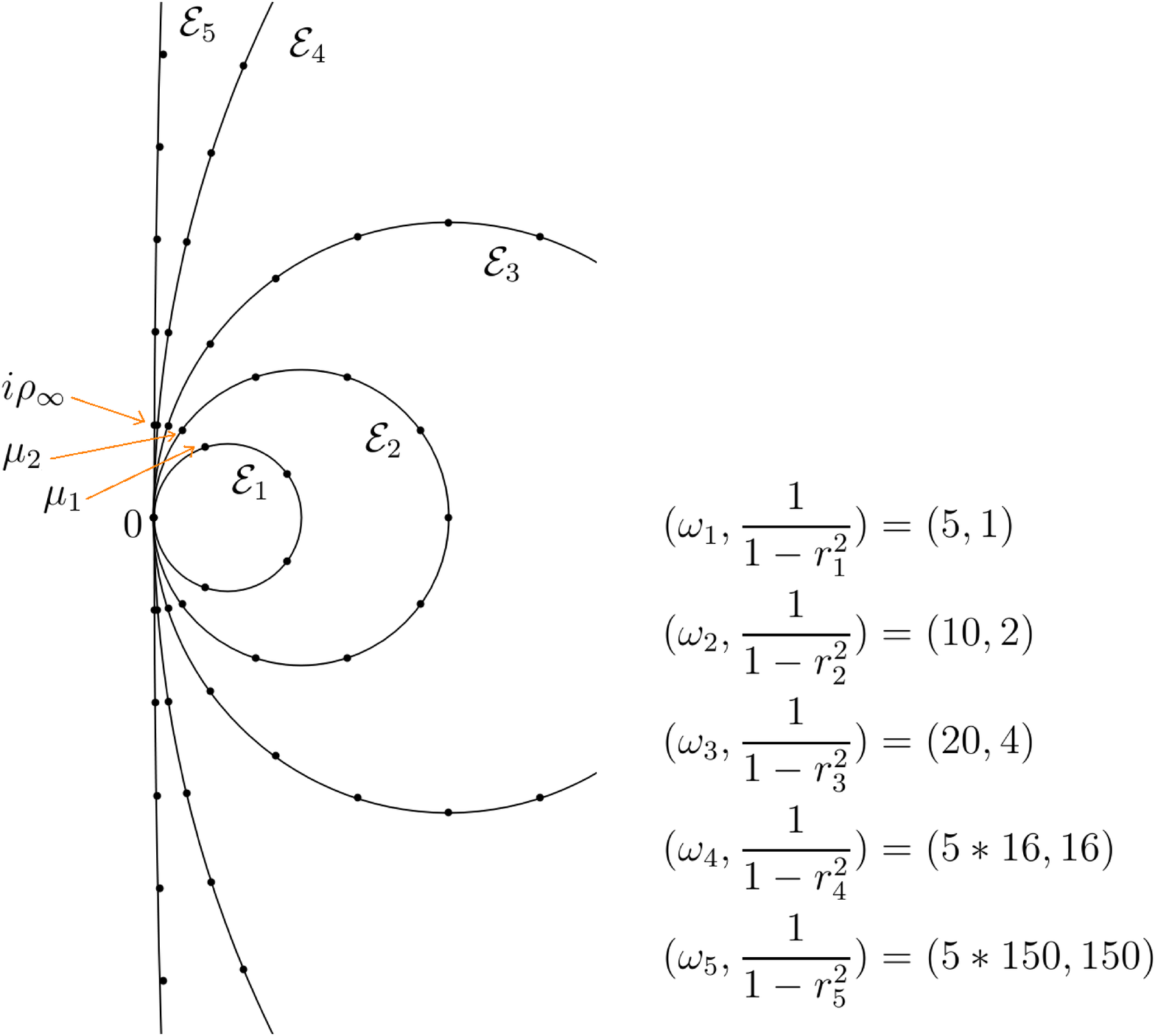}
\end{center}
\caption[ $\EE_n$]{ $\EE_n$: when $r_n \to 1$, the circles degenerate to the vertical axis (see $\EE_5$). Using the notation $\mu_n = \dfrac{1-e^{2\pi i/\omega_n}}{1-r_n^2}$, 
we see that $\mathcal{E}_n$ tends to $<i \rho_{\infty}>$ in the Chabauty topology. }
\label{en}
\end{figure}

\subsection{ The hyperbolic case }

We now consider a sequence $\SSS=\{\SSS_n\}$, where for each $n\in \NN$
$$\SSS_n = <H_{(\theta_1)_n,(\theta_2)_n,a_n}>$$
with $(\theta_1)_n,(\theta_2)_n \in [0,2\pi]$, $\Delta_n=(\theta_2)_n-(\theta_1)_n \in (0,2\pi)$ and $a_n>1$.
In analogy with the elliptic case, we will see in the next proposition that the limit of the sequence $\SSS$ is governed by the limits of
$(\theta_1)_n$, $(\theta_2)_n$, $a_n$ and 
$\rho_n=|\nu_n|$ where $\nu_n = \frac{a_n-1}{1-e^{i\Delta_n}}$.
As before, we can always assume that these sequences converge.

The different limits $\SSS$ can have are summarized in Proposition \ref{hyplimchab} below.

\begin{prop} 
\label{hyplimchab}
The following table shows the Chabauty limit of $\SSS$ according to the possible limits of $(\theta_1)_n$, $(\theta_2)_n$, $a_n$ and 
$\rho_n=|\nu_n|=\frac{a_n-1}{|1-e^{i\Delta_n}|}$.

\begin{tabular}{|c|c|c|c|} 
   \hline
   \backslashbox{$\Delta_{\infty}$}{$a_{\infty}$} & $\infty$ & $1 < a_\infty < \infty$ & $1$  \\
   \hline
   $\Delta_\infty \in (0,2\pi)$ & $\{1\}$ & $H_{(\theta_1)_\infty,(\theta_2)_\infty,a_\infty}$ & subgroup of all hyperbolic elements \\
    &  &  & fixing $(\theta_1)_{\infty}$ and $(\theta_2)_{\infty}$ \\
   \hline
   $\Delta_\infty = 0$, $2 \pi$ & $\{1\}$ & $\{1\}$ & $<P_{p_\infty, \rho_\infty}>$ \\
   \hline
\end{tabular} \\ 

Here the notations are $(\theta_i)_\infty=\lim (\theta_i)_n$, $a_{\infty} = \lim a_n$, $\Delta_{\infty} = \lim \Delta_n$ and $\rho_\infty=\lim \rho_n\in [0,\infty]$ 
with the same convention for $<P_{p, \infty}>$ and  $<P_{p, 0}>$ as in Proposition \ref{elllimchab}
\end{prop} 
\begin{proof}

As above, let us first assume that $a_\infty=1$ and $\Delta_\infty\in \{0,2\pi\}$; define $\theta_\infty=(\theta_1)_\infty=(\theta_2)_\infty \mod 2\pi$.

Also, set $\FF_n := \{ \dfrac{{a_n}^k -1}{1- e^{i \Delta_n}} \,|\, k \in \ZZ \}$
and define the maps $\varphi_n, \varphi: \CC \to \PC$ by 
\begin{align*}
\varphi_n(z) &=(z(1-e^{i \Delta_n}) + 1)^{-1/2} \begin{pmatrix} 1 - z e^{i \Delta_n} & z e^{i \theta_2} \\- z e^{-i \theta_1}& 1+ z \end{pmatrix}  \\
\varphi(z) &= \begin{pmatrix} 1 - z & z e^{i \theta_{\infty}} \\- z e^{-i \theta_{\infty}}& 1+ z \end{pmatrix}
\end{align*}
One should note that the members of $\FF_n$ are the complex numbers $``\nu"$ for the matrices $H^k_{(\theta_1)_n, (\theta_2)_n, a_n}$. 

Then by construction $\SSS_n=\varphi_n(\FF_n)$, and
$\varphi$ is continuous and proper. 

\begin{lem}
$\varphi_n$ converges to $\varphi$ uniformly on every compact subset.
\end{lem}
\begin{proof}
The same argument as in the proof of Lemma \ref{lemell1} works by simply replacing $1-(1-r_n^2)z$ by $z(1-e^{i \Delta_n}) + 1$. 
\end{proof}

\begin{lem}
 For any compact subset $K$ of $\PR$, the closed subset of $\CC$
\[
\overline{\bigcup_{n\geq N}\varphi_n^{-1}(K)}
\]
is compact for $N$ large enough.
\end{lem}
\begin{proof}
The same argument as in the proof of Lemma \ref{lemell2} works by simply replacing $|1-(1-r_n^2)z|^{-1/2} |1-z|$ by $|z(1-e^{i \Delta_n}) + 1|^{-1/2} |1-ze^{i \Delta_n}|$. 
\end{proof}

\begin{lem} In the Hausdorff topology on $\CC$, the sequence of sets $\FF_n$ converges to the set $\FF=\{k i \rho_{\infty} ; k \in \ZZ\}$, 
where $\rho_{\infty} = \lim_{n \to \infty}  |\nu_n| $.
\end{lem} 
\begin{proof}
Note that $\lim_{n \to \infty} \nu_n = i \rho_{\infty}$. This can easily be proved in a direct manner, see Figure \ref{fn} for geometric intuition.
\end{proof}

Here again, the image of $\{k i \rho_{\infty} \,|\, k \in \ZZ\}$ under $\varphi$ is simply $<P_{p_\infty, \rho_\infty}>$.

Thus we are done for the case where $a_\infty=1$ and $\Delta_\infty\in \{0,2\pi\}$. 

The other cases, easier and similar, are left to the reader.
\end{proof}

\begin{figure}[ht]
\begin{center}
\includegraphics[scale=0.3]{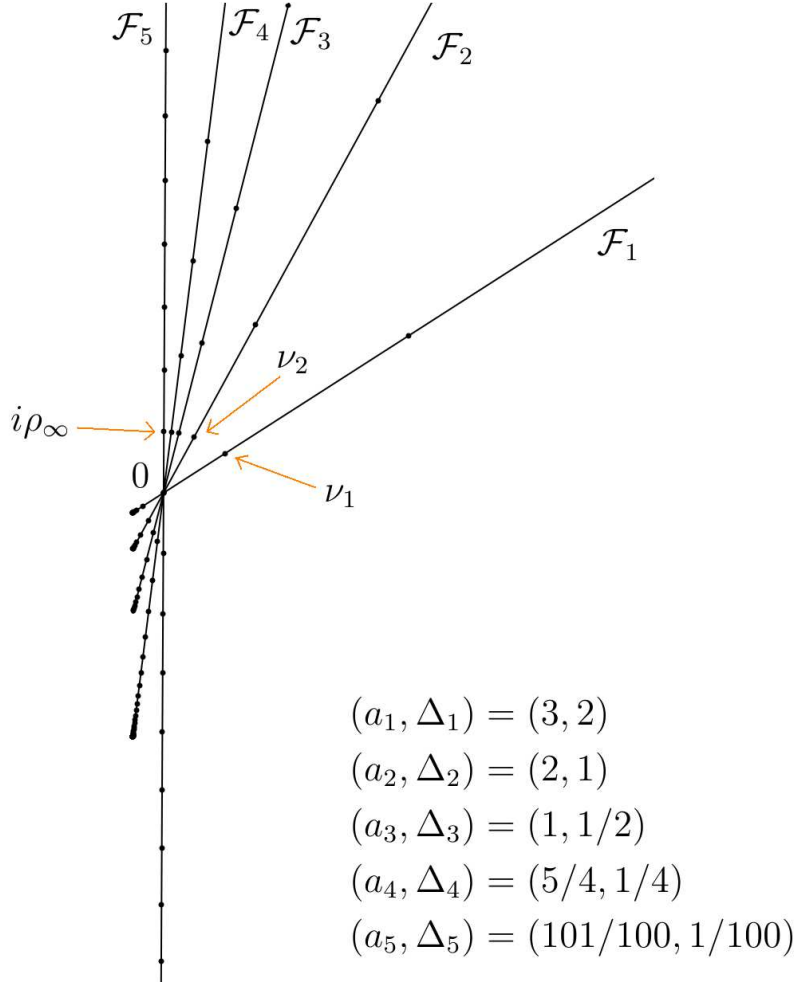}
\end{center}
\caption[$\FF_n$]{ $\FF_n$: when $\Delta_n \to 0$,  the lines containing $\FF_n$ approach the vertical line, and the points on these lines are more and more equally spaced (see $\FF_5$).  }
\label{fn}
\end{figure}

\subsection{ The parabolic case }
\label{pcase}
 In Sections 6.1 and 6.2, we saw that there are subgroups with parabolic generators on the boundary of the space of subgroups with elliptic/hyperbolic generators. 
In the elliptic case, $\rho_\infty =  \lim_{n \to \infty}  |\mu_n| $ indicated which subgroup with parabolic generator was the wanted limit in the Chabauty topology.
In the hyperbolic case, $\rho_\infty = \lim_{n \to \infty} |\nu_n| $ was the good parameter. 

We are interested now in studying the sequences $\SSS=(\SSS_n)$ of subgroups $\SSS_n=<P_{\theta_n, \rho_n} >$, where
\[
P_{\theta_n, \rho_n} = \begin{pmatrix} 1 - i \rho_n & i \rho_n e^{i\theta_n} \\ -i \rho_n e^{-i \theta_n} & 1 + i \rho_n \end{pmatrix}
\]
Set $\theta_\infty = \lim \theta_n$ and $\rho_\infty = \lim \rho_n$. 
 The following proposition is straightforward and its proof is essentially the same as the one for $\RR$ given in Subsection \ref{ss22}
(one may use the Reduction lemma to reduce this case to the case of $\RR$; we left this to the interested readers).  

\begin{prop} There are three cases. 
\begin{itemize}
\item If $\rho_\infty = 0$, then $\SSS$ converges to the group consisting of all parabolic isometries fixing $e^{i \theta_\infty}$.  
\item If $0 < \rho_\infty < \infty$, then $\SSS$ converges to the group generated by $P_{\theta_\infty, \rho_\infty}$. 
\item If $\rho_\infty = \infty$, then $\SSS$ converges to the trivial group. 
\end{itemize} 
\end{prop}  

In particular, the subgroups with one parabolic generator cannot accumulate to the subgroups with elliptic or hyperbolic generators.

Now we are ready to see how the spaces $\ES$, $\HS$, $\PS$ look like. 
 
\pagebreak

\section{Picture of $\ES$} 
\label{es}
 Let us recall some notations from Section \ref{overview}.
For $n = \{2, 3, \ldots, \infty\}$, $D_n$ denotes a copy of the open unit disk, such that each point $p$ of $D_n$ represents the group generated by the rotation around $p$ of order $n$ in $\DD$.
In other words, $D_n$ is the space of subgroups with one elliptic generator of order $n$.
The elements of $D_{\infty}$ are those subgroups consisting of all rotations around the given point in $\DD$.

We are now going to \textit{bend} these disks, by requiring that, whenever a point of a disk $D_n$ is at altitude $x_3>0$, then its parameter $\mu$ verifies $|\mu|=x_3$. 
More precisely, represent $D_n$ in $\RR^3$ by the image of $\DD$ by the map
\begin{align*}
 \DD & \rightarrow \RR^3 \\
 p =r e^{i\theta} & \mapsto \big(r \cos \theta, r \sin \theta, \frac{|1-e^{2\pi i /n}|}{1-r^2}\big)
\end{align*}
Note that this image blows up when $p$ approaches the boundary of $\DD$.

Also, identify $D_{\infty}$ with the open unit disk in the $(x,y)$-plane,
and identify the cylinder 
\[
 \{(e^{i\theta},\rho)\,|\,  \rho\geq0\}
\]
with $\PS$ by asking $(e^{i\theta},\rho)$ with $\rho>0$ to represent the subgroup generated by the parabolic element fixing $e^{i\theta} \in S^1$ and having the form $z\mapsto z + 2\rho $ in our normalization (see Section \ref{matrix}).
Of course then, every element $(e^{i\theta},0)$ of the unit circle in the $(x,y)$-plane represents the subgroup of all parabolic elements fixing $e^{i\theta}$.
Finally, the identity group should be identified with a point at infinity of coordinates $(0,0,\infty)$.

The reader is invited to check, using the results of Section \ref{limchab}, that 
whenever a sequence of points $p_n$ in this picture converges to some $p_\infty$, then the subgroups represented by $p_n$ converge in the Chabauty topology to the subgroup represented by $p_\infty$.

\begin{figure}[ht]
\begin{center}
\includegraphics[scale=0.3]{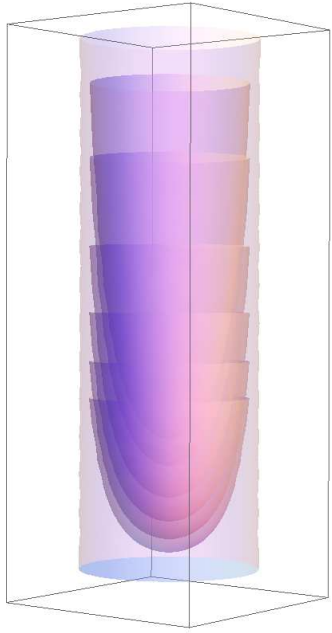}
\end{center}
\caption[The space of subgroups with one elliptic generator]{The leaves $D_n$ accumulate to $D_\infty$ and $\PS$. Here $n = 2, \ldots, 7$. }
\label{lamination}
\end{figure}

\pagebreak

\section{Picture of $\HS$} 
\label{hs}

In Section \ref{overview}, we explained how $\HS$ has in its interior a cone on the open M\"{o}bius band.
This cone was foliated by open M\"{o}bius bands in the obvious way.

We are going to bend the leaves $M_a$ of this foliation, like in Section \ref{es}, by requiring that a point of at altitude $x_3>0$ always represents a subgroup generated by a hyperbolic element with parameter $\nu$ verifying $|\nu|=x_3$. 
More precisely, for any $a>0$, represent $M_a$ in $\RR^3$ by the image of the upper left triangle $M$ of $[0,2\pi]^2$ that parametrizes an open M\"{o}bius band (see Figure \ref{mobiusband} in Section \ref{overview}), by the map
\begin{align*}
 M & \rightarrow \RR^3 \\
 (\theta_1,\theta_2) & \mapsto \big(\theta_1, \theta_2, \frac{a-1}{|1-e^{i(\theta_2-\theta_1)}|}\big)
\end{align*}

Also, define $M_0$ to be a copy of $M$ in the $(x,y)$-plane. Every element $(\theta_1,\theta_2)$ of $M_0$ represents the subgroup of all hyperbolic elements having the same two fixed points $e^{i\theta_1},e^{i\theta_2}\in S^1$.

Notice that the boundary of a M\"{o}bius band is a circle, here parametrized for all $a$ by the diagonal 
\[
 \partial M_a=\{(\theta,\theta,a)\,|\, \theta\in [0,2\pi] \}
\]
where $(0,0)$, $(0,2\pi)$ and $(2\pi,2\pi)$ are identified.

As for the elliptic case, note that each leaf blows up to infinity when $(\theta_1,\theta_2)$ approaches the diagonal $\partial M_a$.

Now, identify the set 
\[
 \{((\theta,\theta,\rho)\,|\, \theta\in [0,2\pi],\; \rho\geq 0\}
\]
with $\PS$ by asking $(\theta,\theta,\rho)$ with $\rho>0$ to represent the subgroup generated by the parabolic element fixing $e^{i\theta} \in S^1$ and having the form $z\mapsto z+ 2\rho $ in our normalization (see Section \ref{matrix}).
Of course then, every element $(\theta, \theta, 0)$ of $M_0$ represents the subgroup of all parabolic elements fixing $e^{i\theta}$.
Finally, the identity group is again identified with the point at infinity $(0,0,\infty)$.

As before, the reader is invited to check, using the results of Section \ref{limchab}, that 
whenever a sequence of points $p_n$ in this picture converges to some $p_\infty$, then the subgroups represented by $p_n$ converge in the Chabauty topology to the subgroup represented by $p_\infty$.

\begin{figure}[ht]
\begin{center}
\includegraphics[scale=0.33]{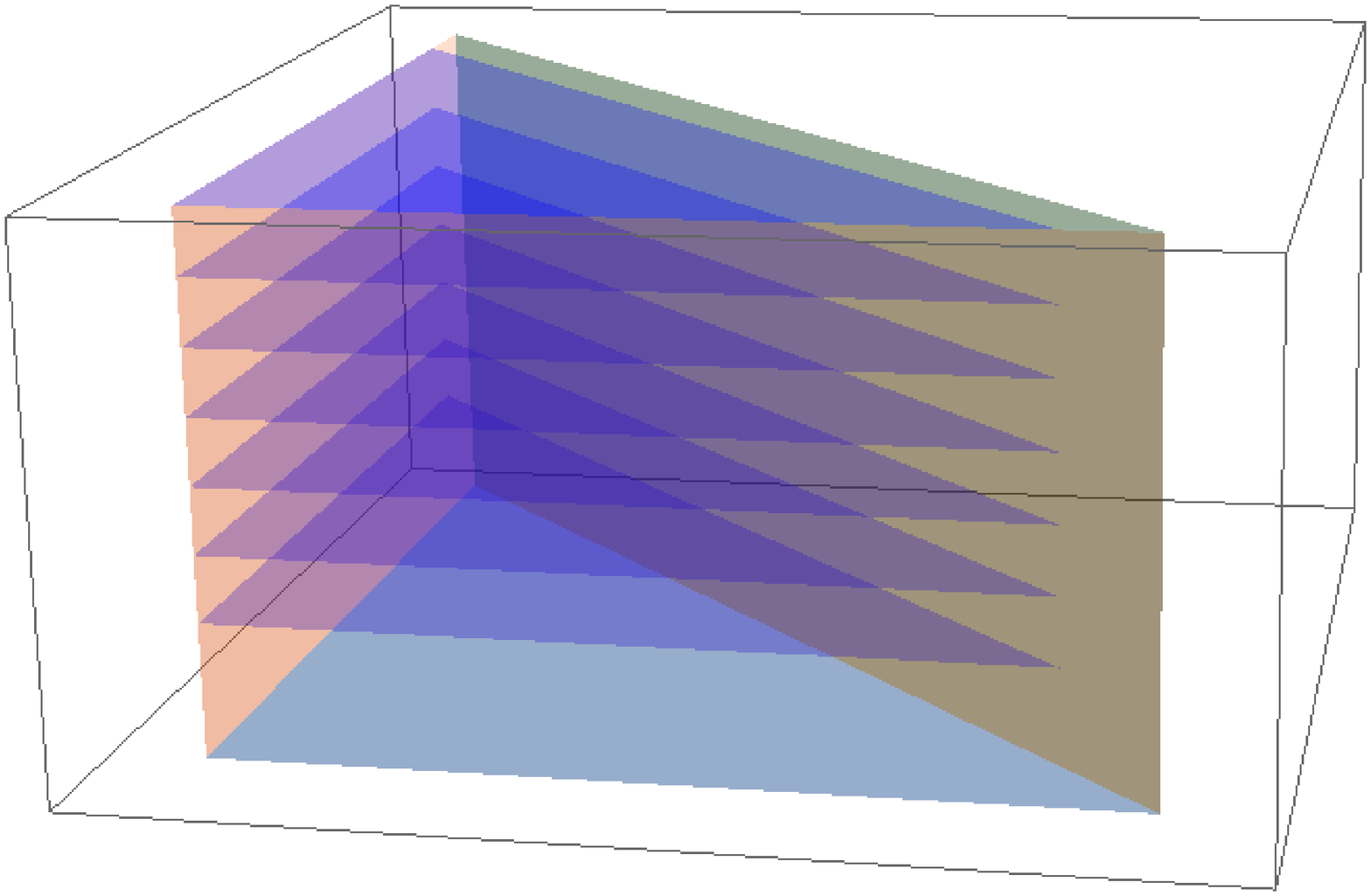}
\includegraphics[scale=0.33]{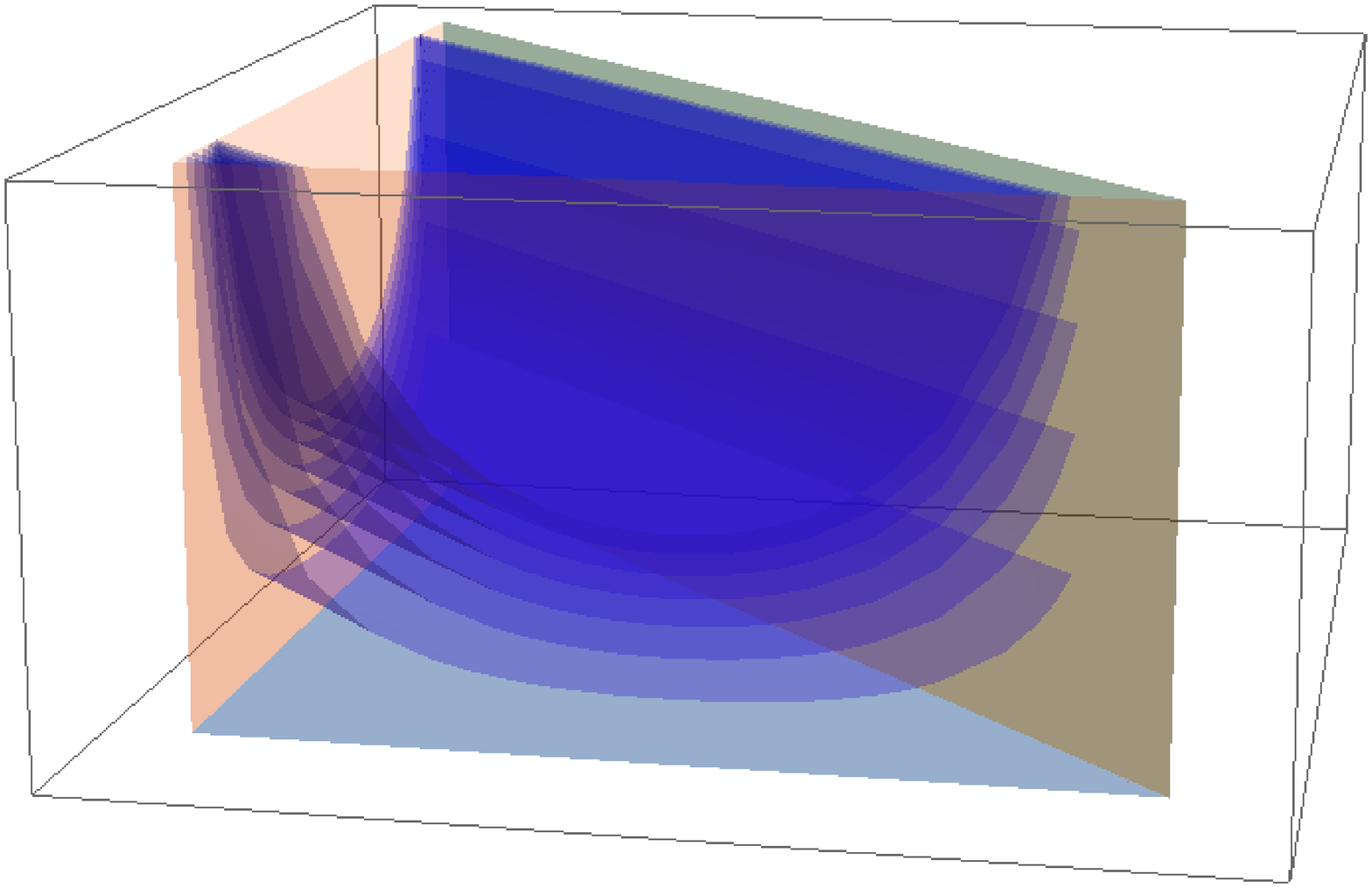}
\end{center}
\caption[The space of subgroups with one hyperbolic generator (correct version)]{The picture on the left shows the foliation by leaves of constant translation length in the original prism. The leaves are parallel triangles (or after gluing: M\"{o}bius bands). The picture on the right shows the same leaves after bending. To obtain a correct picture, the front and the left walls (colored in orange) should be identified so that each leaf becomes a M\"{o}bius band. 
Note that, for visibility reasons, we only drew a few leaves of the foliation. }
\label{foliations}
\end{figure}  


\pagebreak

\section{Gluing $\ES$ and $\HS$} 
\label{glu}
The last thing we have to do to complete the description of $\CS$ is to glue the spaces $\ES$ and $\HS$ along $\PS$.
In view of the bending we performed in Sections \ref{es} and \ref{hs}, it should be clear now that the correct gluing map is


\begin{align*}
\Phi: \PS = S^1 \times \RR^+ \subset \ES &\rightarrow \PS = \partial M_0 \times \RR^+ \subset \HS \\
 (e^{i \theta}, c) &\mapsto ( \theta, \theta, c) 
\end{align*}

This gives us the closure of the space of all subgroups of $\PR$ with one generator in the Chabauty topology.

 We finally obtained the main theorem.
 \begin{thm}[Main Theorem] 
 \label{mainthm} 
 The space of all geometric limits of closed subgroups of $\PR$ with one generator is $\CS = \ES \cup \HS / \sim$, where
\begin{itemize}
 \item[(1)] $\ES$ is a wedge sum of countably many 2-spheres $D_n / \partial D_n $, which accumulate to a disk $D_\infty$ and to the cone $\PS$ on the circle $\partial D_\infty$.
(see Figure \ref{lamination})
 \item[(2)] $\HS$ is the cone on a closed M\"{o}bius band, the inside of which is foliated by ``bent'' open M\"{o}bius bands, which accumulate to an open M\"{o}bius band $M_0$ and to the cone $\PS$ on the circle $\partial M_0$ (see Figure \ref{foliations}).
 \item[(3)] $\sim$ represents the gluing of $\ES$ and $\HS$ along $\PS$.
\end{itemize}

See Sections \ref{es} and \ref{hs} for complete parametrizations of $\ES$ and $\HS$. 
\end{thm} 

It seems worth to mention some easy corollaries of this theorem which tell us about the topology of $\CS$. Some of them may not be of any special interest in the perspective of the geometric limit of Kleinian groups, but could be interesting in purely topological point of view. This is to be compared with the 1-connectivity of the Chabauty space of $\RR^n$ in \cite{Benoit}.

\begin{cor}
\label{1conn} 
$\CS$ is simply-connected.
\end{cor} 
\begin{proof} Since $\HS$ is contractible, $\CS$ is homotopy equivalent to the wedge sum of countably many 2-spheres. 
\end{proof}

In particular, the path-connectivity of $\CS$ tells us that we can continuously deform any group to the any other group in $\CS$. 

\begin{cor} $\pi_2(\CS) \cong  H^2(\CS) \cong \bigoplus\limits_{i \in \NN} \ZZ$.
\end{cor} 
\begin{proof} By Corollary \ref{1conn}, we know $\CS$ is 1-connected. Hence the result follows from the Hurewicz theorem. 
\end{proof} 

\begin{cor} Let $\CS' = \{ \Gamma \in \CS : \Gamma \mbox{ is discrete.}\}$. Then $\CS'$ is still simply-connected. 
\end{cor} 

This says the connectivity result of $\CS$ does not depend on the part corresponding to continuous groups. 





\pagebreak

\section{Future work and related topics} 
\label{future}
 We have studied the space of geometric limits of the one-generator closed subgroups of $\PR$. There are two obvious ways to generalize this; 
 one can study either the space of geometric limits of two-generator closed subgroups of $\PR$ or the space of geometric limits of one-generator closed subgroups of $\PC$. 

The former case has some intricate features, but the latter one involves much more diverse phenomena. For instance, subgroups generated by one hyperbolic element 
can converge in the Chabauty topology to a subgroup generated by \emph{two} parabolic elements. 
The authors are currently writing a paper about the space of one-generator closed subgroups of $\PC$. In this upcoming paper, 
we use similar parametrizations, and obtain results comparable to those we obtained in Section \ref{limchab}. But the global picture is emphatically more complicated. 
Relatively easy cases will be explored by understanding the Chabauty topology of $\CC^*$ and applying the Reduction Lemma. 
On the way, we will encounter an interesting space which we call the Hubbard's cabbage. 
Much more is involved for the full generality.

One remote but major goal of this project is to understand the global topology of the space of the type-preserving quasifuchsian representations of the punctured-torus group. For the definitions and detailed theory, we refer the readers to \cite{puncturedtorus}. 
In this case, we are interested in the subgroups of $\PC$ with two hyperbolic generators. Thus the boundary is much more complicated than the one of the space of one-generator subgroups of $\PC$.  
The major complication of the boundary  comes from the enrichment phenomenon, specific to geometric limits of Kleinian groups (see \cite{Hubb2}).
Considering all possible geometric limits in the one-generator case, however, one can have a much better understanding how the enrichment occurs in more explicit terms. 
With a deep understanding on the enrichment of Kleinian groups, one might hope to attack the following conjecture due to Bowditch. See, for instance \cite{Minsky1}. 
\begin{conj} 
In the space of the type-preserving representations of the punctured-torus group, the representations satisfying the $BQ$-condition are precisely the quasifuchsian representations. 
\end{conj} 

Consider the trace of the commutator of the generators in the character variety $\chi(F_2)$. This defines an Out($F_2$)-invariant function $\varphi$ on $\chi(F_2)$. Then the level set $\varphi^{-1}(-2)$ is the slice consists of the type-preserving representations of the punctured-torus groups (here, `type-preserving' means the commutator of the generators is parabolic).  One more term needed to be defined here is the $BQ$-condition.  
\begin{defn}
$[\rho] \in \chi(F_n)$ is said to satisfy the BQ-condition if 
\begin{itemize}
\item[(1)] $\rho(X)$ is hyperbolic for all primitive element $x \in F_2$.
\item[(2)] The number of conjugacy classes of primitive elements $x$ such that $|tr(\rho(x))| \le 2$ is finite. 
\end{itemize} 
\end{defn} 

We would like to know the global topology of this space, namely describe how the boundary looks like explicitly. We hope that the techniques we have been developed are potentially useful.


\bibliographystyle{alpha}
\bibliography{biblio}

\end{document}